\documentclass{actams-en}
\usepackage{CJK}
\usepackage{harpoon}
\usepackage{threeparttable}
\usepackage{tabularx,multirow,bigdelim}
\usepackage{algorithm,algorithmic}
\usepackage{graphics,graphicx,subfigure}

\numberwithin{equation}{section} 
\def\ds{\displaystyle }
\newcommand{\p}{\partial}
\usepackage{cases}
\def\d{\mathrm{d}}
\usepackage{booktabs}
\usepackage{color,xcolor}

\begin{document}

 \Volume{2024}{}{}{}
 \DOI{}
 \PageNum{1}
\EditorNote{The first author is supported by the National Social Science Foundation of China (Grant No.24BTJ006).}

\EditorNote{$^*$Corresponding author}

\abovedisplayskip 6pt plus 2pt minus 2pt \belowdisplayskip 6pt
plus 2pt minus 2pt
\def\no{\nonumber}
\def\R{{\Bbb R}}
\newenvironment{prof}[1][Proof]{\indent\textbf{#1}\quad }
{\hfill $\Box$\vspace{0.7mm}}
\def\q{\quad} \def\qq{\qquad}
\allowdisplaybreaks[4] 

\AuthorMark{Meihui Zhang et al.
}
\TitleMark{\uppercase{Generalizing subdiffusive Black-Scholes model}}

\title{\uppercase{Generalizing subdiffusive Black-Scholes model by variable exponent: Model transformation and numerical approximation}}
\begin{CJK*}{GBK}{kai}
\author{\sl{Meihui ZHANG}$^1$\quad\sl{Yaxue LIU}$^1$\quad\sl{Mengmeng LIU}$^{2,3*}$\quad\sl{Wenlin QIU}$^3$\quad\sl{Xiangcheng ZHENG}$^4$}
   {  1. School of Statistics and Mathematics, Shandong University of Finance and Economics, Jinan {250014}, China\\
    E-mail\,$:$ zmh$\_$1212@163.com;~liuyaxue@sdufe.edu.cn\\
2. School of Mathematics and Statistics, Hunan Normal University, Changsha, Hunan {410081}, China\\
    E-mail\,$:$ liumengmeng423@163.com  \\
3. School of Mathematics, Shandong University, Jinan {250100}, China\\
    E-mail\,$:$ wlqiu@sdu.edu.cn\\
4. School of Mathematics, State Key Laboratory of Cryptography and Digital Economy Security, Shandong University, Jinan {250100}, China\\
    E-mail\,$:$  xzheng@sdu.edu.cn}
\maketitle
\end{CJK*}

\Abstract{This work generalizes the subdiffusive Black-Scholes model by introducing the variable exponent in order to provide adequate descriptions for the option pricing, where the variable exponent may account for the variation of the memory property. In addition to standard nonlinear-to-linear transformation, we apply a further spatial-temporal transformation to convert the model to a more tractable form in order to circumvent the difficulties caused by the ``non-positive, non-monotonic'' variable-exponent memory kernel. An interesting phenomenon is that the spatial transformation not only eliminates the advection term but naturally turns the original noncoercive spatial operator into a coercive one due to the specific structure of the Black-Scholes model, which thus avoids imposing constraints on coefficients. Then we perform numerical analysis for both the semi-discrete and fully discrete schemes to support numerical simulation. Numerical experiments are carried out to substantiate the theoretical results.
}
\Keywords{Black-Scholes model; Subdiffusion; Variable exponent; Error estimate; Option pricing
}
\MRSubClass{65N06; 65N12; 65N30
}

\section{Introduction}
The Black-Scholes equation is a classical model for option pricing, which takes the form as \cite{BlaSch}
\begin{equation}\label{IBS}
\p_t v(S,t)+\frac{1}{2}\sigma^2S^2\p_S^2v(S,t)+rS\p_Sv(S,t)-rv(S,t)=0,
\end{equation}
where $v(S,t)$ is the price of the option, $r\geq 0$ is the risk-free rate and $\sigma > 0$ is the volatility. There exist extensive investigations for this model, see e.g. \cite{Cap} and the references therein.
Despite its wide and successful applications, it is shown that it fails to characterize the significant movements or jumps over small time
steps in a financial market \cite{Car}.
A potential way to resolve this issue is to replace the integer-order derivative $\p_tv$ in (\ref{IBS}) by the fractional derivative $^R\p^\alpha_t v$ with $\alpha\in (0,1)$  \cite{Jinbook}
\begin{equation}\label{FDI}
^R\p^\alpha_t g(t):=\p_t\int_t^T\frac{g(s)-g(T)}{\Gamma(1-\alpha)(s-t)^\alpha}\d s,
\end{equation}
and some progresses on the resulting subdiffusive Black-Scholes model have been reached in the past few decades \cite{Kazmi,add6,Shi3,Fazlo,Wys,Zeng,ZhaLiu}.
  Nevertheless, a constant exponent may not suffice to  accommodate the variation of the memory property caused by, e.g. the uncertainties or fluctuations in the financial market \cite{Che2,De,Duan,add1,GCH,Jum,Skov}. A possible remedy is to introduce the  variable exponent. In \cite{Luc}, a variable-exponent fractional option pricing model is applied to accommodate various phenomena in financial activities such as the seasonal changes and temporary crises, and the effectiveness of the variable-exponent model has been demonstrated based on the real option price data. Thus, we consider the following subdiffusive Black-Scholes model with variable exponent $0<\alpha(t)<1$ and the expiry $T>0$
\begin{equation}\label{FBS}\begin{array}{c}
\ds ^R\p^{\alpha(t)}_tv+\frac{1}{2}\sigma^2S^2\p_S^2v+rS\p_Sv-rv=0, ~~(S,t)\in (a,b)\times (0,T);\\[0.05in]
 \ds v(a,t)=c_l(t),~~v(b,t)=c_r(t),~~v(S,T)=c_t(S).
\end{array}
\end{equation}
Here  $c_l$, $c_r$, $c_t$ are given data and  $^R\p^{\alpha(t)}_t$ with $0<\alpha(t)<1$ is the variable-exponent fractional differential operator \cite{LorHar}
\begin{equation}\label{VCfrac}^R\p^{\alpha(t)}_t g(t):=\p_t\int_t^T\frac{g(s)-g(T)}{\Gamma(1-\alpha(s-t))(s-t)^{\alpha(s-t)}}\d s.
\end{equation}

As the exact solutions of model (\ref{FBS}) are in general not available, it is necessary to study numerical methods. Extensive numerical results for model (\ref{FBS}) (or its variant (\ref{VtFDEs}) by standard logarithmic transformation) with $\alpha(t)\equiv \alpha$ for some constant $0<\alpha<1$ have been obtained, while there are rare studies for the variable-exponent case.  A main difficulty is that the variable-exponent memory kernel is non-positive and non-monotonic such that conventional numerical analysis methods do not apply. In a recent work \cite{ZhaZhe}, a relevant variable-exponent fractional Black-Scholes model is considered, where an additional term $\p_t v$ is included in (\ref{FBS}). Due to this additional leading term, the impact of the variable exponent is significantly weakened such that the numerical analysis could be performed.

For model (\ref{FBS}) (or its variant (\ref{VtFDEs})) where the variable-exponent term serves as the leading term, one could apply the convolution method developed in \cite{Zhe} to convert the original model to a more tractable form. Nevertheless, the convolution method could lead to the coupling between spatial operators and the temporal convolution such that the advection term in (\ref{VtFDEs}) may cause difficulties in numerical analysis.  Thus, a further spatial transformation could be employed for (\ref{VtFDEs}) to eliminate the advection term, cf. the transformed equation (\ref{eq4}). In general, this method changes the spatial coefficients such that some constraints on these coefficients should be imposed to ensure the coercivity of the spatial operators. However, an interesting phenomenon is that the spatial operator in the transformed equation (\ref{eq4}) naturally keeps coercive due to the specific structure of (\ref{VtFDEs}), even though the spatial operator in (\ref{VtFDEs}) may not be coercive. We thus combine this and the convolution method to finally obtain a numerically feasible scheme (\ref{eq13})--(\ref{eq14}).

We are now in the position to consider numerical approximation to the transformed model (\ref{eq13})--(\ref{eq14}). The piecewise linear interpolation approximation is used to discretize the convolutions, which results in a time-discrete scheme. Then the finite element method is applied in space to construct the fully discrete scheme. We derive error estimates for both schemes and perform numerical experiments to substantiate the theoretical findings.

The rest of the work is organized as follows: Section \ref{Sec2} presents the model transformation, which reduces the original model into a more tractable form. Then, Section \ref{Sec3} establishes a time-discrete scheme and analyzes its numerical stability and error estimate. In Section \ref{Sec4}, a fully discrete finite element scheme is constructed and analyzed. Section \ref{Sec5} provides some numerical results to validate our theoretical results. Finally, conclusions are given in Section \ref{Sec6}.

\section{Model transformation} \label{Sec2}

\subsection{Nonlinear-to-linear transformation}\label{sec2.1}
We follow the standard method, see e.g. \cite{ZhaLiu,ZhaZhe}, to convert the nonlinear model (\ref{FBS}) to a linear version.
By the variable substitution $t=T-\tau$ we obtain from (\ref{VCfrac}) that
\begin{equation}\label{VCtrans}
\begin{array}{rl}
\ds ^R\p^{\alpha(t)}_t g(t)&\ds=\p_t\int_t^T\frac{g(s)-g(T)}{\Gamma(1-\alpha(s-t))(s-t)^{\alpha(s-t)}}\d s\\[0.15in]
&\ds=-\p_\tau\int_{T-\tau}^T\frac{g(s)-g(T)}{\Gamma(1-\alpha(s-T+\tau))(s-(T-\tau))^{\alpha(s-T+\tau)}}\d s\\[0.15in]
&\ds={-\p_\tau}\int_{0}^\tau\frac{g(T-\theta)-g(T)}{\Gamma(1-\alpha(\tau-\theta))(\tau-\theta)^{\alpha(\tau-\theta)}}\d \theta.
\end{array}
\end{equation}
{\color{black}{Then we intend to employ the transformation $x=\ln S$ and accordingly define $u(x,t)=v(e^x,T-t)$. For the typical case where $S\in (a,b)=(0,\infty)$, the logarithmic transformation will lead to an unbounded domain. In this case, we follow the standard treatment, see e.g. \cite{add1,ZhaLiu}, to truncate the unbounded domain into a finite interval. Thus, the subdiffusive Black-Scholes model (\ref{FBS}) could be transformed into the following form}}
\begin{equation}\label{FBSt}\begin{array}{l}
 \ds - \p^{\alpha(t)}_tu(x,t)+\frac{1}{2}\sigma^2\p_x^2u(x,t)\\
 \ds\qquad\qquad+\Big(r-\frac{\sigma^2}{2}\Big)\p_xu(x,t)-ru(x,t)=0, ~~(x,t)\in (\bar a,\bar b)\times (0,T); \\[0.1in]
 \ds u(\bar a,t)=\bar c_l(t),~~u(\bar b,t)=\bar c_r(t),~~u(x,0)=\bar c_t(x)
\end{array}
\end{equation}
{\color{black}where $(\bar a,\bar b)$ is defined as $(\ln a,\ln b)$ if $(\ln a,\ln b)$ is bounded or the truncated interval if $(\ln a,\ln b)$ is unbounded}, $\bar c_l(t):=c_l(T-t)$, $\bar c_r(t):=c_r(T-t)$, $\bar c_t(x):=c_t(e^x)$ and $\p^{\alpha(t)}_t$ is defined by
\begin{equation}\label{VOD}
\p^{\alpha(t)}_tu(x,t):=\p_t\int_{0}^t\frac{u(x,s)-u(x,0)}{\Gamma(1-\alpha(t-s))(t-s)^{\alpha(t-s)}}\d s.
\end{equation}
Define
$$k(t):=\frac{t^{-\alpha(t)}}{\Gamma(1-\alpha(t))}$$
 such that (\ref{VOD}) could be reformulated via the symbol of convolution
\begin{equation}\label{zz1}
\p^{\alpha(t)}_tu(x,t)=\p_t\big(k(t)*(u(x,t)-u(x,0))\big),
\end{equation}
where $*$ denotes the convolution in time defined as
$$g_1(t)*g_2(t):=\int_0^tg_1(t-s)g_2(s)\d s. $$
Then a direct calculation in (\ref{zz1}) leads to
\begin{equation}\label{zz2}
\p^{\alpha(t)}_tu(x,t)=k(t)*\p_tu(x,t).
\end{equation}
Furthermore, by the homogenization technique, which replaces $u$ in (\ref{FBSt}) by
$$u-\frac{\bar b-x}{\bar b-\bar a}\bar c_l(t)-\frac{x-\bar a}{\bar b-\bar a}\bar c_r(t),$$
 one could reduce (\ref{FBSt}) to its homogeneous boundary-value analogue with an additional right-hand side term. Therefore, without loss of generality, we invoke this and (\ref{zz2}) to consider the following problem
   \begin{equation}\label{VtFDEs}\begin{array}{l}
 \ds  \p^{\alpha(t)}_tu(x,t)-\frac{1}{2}\sigma^2\p_x^2u(x,t)\\[0.1in]
 \ds\qquad-\Big(r-\frac{\sigma^2}{2}\Big)\p_xu(x,t)+ru(x,t)=f(x,t), ~~(x,t)\in (\bar a,\bar b)\times (0,T); \\[0.125in]
 \ds u(\bar a,t)=u(\bar b,t)=0,~~u(x,0)=\bar c_t(x).
\end{array}
\end{equation}

\subsection{Further spatial-temporal transformation}
We intend to perform a further spatial-temporal transformation to achieve the following goals:
\begin{itemize}
\item[$\bullet$] Convert the variable-exponent factor from the leading term to a low-order term to resolve the difficulties caused by the variable-exponent kernel;
\item[$\bullet$] Eliminate the advection term to facilitate the error estimate.
\end{itemize}
First, we let
\begin{equation*}
u(x,t) = \exp\left( (\frac{1}{2}-\frac{r}{\sigma^2})x\right)\phi(x,t), \quad \sigma>0.
    \end{equation*}
Then, we use the technique of \cite{Wang4} to transform problem \eqref{VtFDEs} into
\begin{numcases}{}
\p_t^{\alpha(t)}\phi(x,t) -  \frac{1}{2}\sigma^2\p_x^2 \phi (x,t) + \lambda \phi (x,t) = \chi(x,t),~(x,t)\in (\bar a,\bar b)\times (0,T), \label{eq4}\\
\phi(\bar a,t)=0, \quad \phi(\bar b,t)=0, \label{eq5}\\
\phi(x,0)=\bar c_t^*(x), \quad x \in(\bar a,\bar b),\label{eq6}
\end{numcases}
where
\begin{equation}\label{eq7}
\begin{aligned}
&\lambda  = \frac{1}{8}\sigma^2+\frac{r^2}{2\sigma^2}+\frac{r}{2}=\frac{1}{2}\Big(\frac{\sigma}{2}+\frac{r}{\sigma}\Big)^2> 0, \quad \sigma>0, \quad r\geq 0,   \\
&\chi(x,t)  = \exp\left(-(\frac{1}{2}-\frac{r}{\sigma^2})x\right)f(x,t),\quad
\bar c_t^*(x) =  \exp\left(-(\frac{1}{2}-\frac{r}{\sigma^2})x\right)\bar c_t(x).
\end{aligned}	
\end{equation}
As $\lambda>0$, the spatial operator is coercive for any choice of the parameters $(\sigma,r)$ due to the specific structure of the Black-Scholes model, which significantly facilitates the numerical computation and analysis. Note that this phenomenon still occurs if we replace the $\p_t^{\alpha(t)}\phi(x,t)$ by standard first-order time derivative $\p_t\phi(x,t)$, which implies its generality.

According to equation \eqref{zz2} and the convolution method proposed in \cite{Zhe}, we perform the following convolution with equation \eqref{eq4}
\begin{equation}\label{xm}
   \beta_{\alpha_0} * \left[ (k*\p_t \phi) -  \frac{1}{2}\sigma^2\p_x^2 \phi + \lambda \phi -\chi \right] = 0,
\end{equation}
thus we further calculate that
\begin{equation}\label{eq9}
    q(t)* \p_t \phi(x,t) -  \frac{1}{2}\sigma^2\beta_{\alpha_0} *  \p_x^2 \phi(x,t) + \lambda\beta_{\alpha_0} * \phi(x,t)  -  \beta_{\alpha_0} * \chi(x,t) = 0,~t>0,
\end{equation}
where $\beta_{\mu}=\frac{t^{\mu-1}}{\Gamma(\mu)}$, $\alpha_0=\alpha(0)$ and the function (see \cite{Zhe})
\begin{equation}\label{eq10}
   \begin{split}
       q(t) : = (\beta_{\alpha_0} * k) (t) = \int_0^1 \frac{\mathcal{G}(t\varsigma)}{\Gamma(1-\alpha(t\varsigma))} \frac{(1-\varsigma)^{\alpha_0-1}\varsigma^{-\alpha_0}}{\Gamma(\alpha_0)} \d \varsigma, \quad q(0)=1,
   \end{split}
\end{equation}
in which $\mathcal{G}(t)=t^{\alpha_0 - \alpha(t)}$. Hence, taking $(q* \p_t \phi)(x,t) = \phi(x,t) + (q' * \phi)(x,t) - q(t)\bar c_t^*(x)$ into \eqref{eq9} yields
\begin{multline}\label{eq11}
       \phi(x,t) + (q' * \phi)(x,t) -  \frac{1}{2}\sigma^2\beta_{\alpha_0} *  \p_x^2 \phi(x,t) + \lambda\beta_{\alpha_0} * \phi (x,t)\\ =  (\beta_{\alpha_0} * \chi)(x,t) + q(t)\bar c_t^*(x), \quad (x,t)\in (\bar a,\bar b)\times (0,T),
  \end{multline}
  \begin{equation}\label{eq12}
 \phi(\bar a,t)=0, \quad \phi(\bar b,t)=0, \quad
\phi(x,0)=\bar c_t^*(x), \quad x \in(\bar a,\bar b).
   \end{equation}
Then we define $\omega(x,t)= \phi(x,t)-\bar c_t^*(x)$ to obtain
\begin{multline}\label{eq13}
       \omega(x,t) + (q' * \omega)(x,t) -  \frac{1}{2}\sigma^2\beta_{\alpha_0} *  \p_x^2 \omega(x,t) + \lambda\beta_{\alpha_0} * \omega(x,t) = F(x,t),\\   (x,t)\in (\bar a,\bar b)\times (0,T),
  \end{multline}
   \begin{equation}\label{eq14}
      \omega(\bar a,t)=0, \quad \omega(\bar b,t)=0, \quad
\omega(x,0) = 0,  \quad x \in(\bar a,\bar b).
   \end{equation}
   where $F(x,t)=\beta_{\alpha_0}(t) * (\frac{1}{2}\sigma^2\p_x^2 \bar c_t^*(x) -\lambda\bar c_t^*(x) + \chi(x,t)) $.
{\color{black}
\begin{remark}
This work considers the case that the volatility $\sigma$ is a constant. Although the nonlinear-to-linear transformation in Section \ref{sec2.1} works for the more general case $\sigma=\sigma(S,t)$, the spatial-temporal transformation in this section will no longer work. A possible remedy for the case $\sigma=\sigma(t)$ is to replace the convolution method used in (\ref{xm}) by the perturbation method proposed in \cite{Zhe}, which only performs a kernel splitting for $k$ without affecting other terms. We will investigate this more general case in the future.

\end{remark}}
\subsection{Assumptions}
 For a positive integer $m$, a real number $1 \le p \le \infty$ and an interval $\mathcal I$, let $W^{m,p}({\mathcal I})$ be the Sobolev space of functions with weak derivatives up to order $m$ in $L^p(\mathcal I)$, where $L^p(\mathcal I)$ refers to the space of $p$th Lebesgue integrable functions on $\mathcal I$. Let  $H^m({\mathcal I}) := W^{m,2}({\mathcal I})$ and $H^m_0({\mathcal I})$ be its subspace with the zero boundary conditions up to order $m-1$. 
 For a Banach space ${\cal X}$, let $W^{m,p}(\mathcal I;\mathcal X)$ be the space of functions in $W^{m,p}(\mathcal I)$ with respect to $\| \cdot \|_{\cal X}$. All spaces are equipped with the standard norms \cite{AdaFou}.
 We set $\|\cdot\|:=\|\cdot\|_{L^2({\mathcal I})}$ for brevity if $\mathcal I$ represents a spatial interval, and drop the notation ${\mathcal I}$ in the spaces and norms if $\mathcal I$ represents a spatial interval, e.g. we denote $\|\cdot\|_{H^2}=\|\cdot\|_{H^2(\mathcal I)}$.

Throughout this work, we consider the smooth variable exponent, i.e., we assume that
 $0<\alpha(t)\leq \alpha^*<1$ for $t\in [0,T]$ and $\alpha\in W^{1,\infty}(0,T)$. Furthermore, we use $C$ to denote a generic positive constant that may assume different values at different occurrences.

Following the regularity results of \cite[Theorems 4.2-4.3]{Zhe}, we assume that the solutions to the model \eqref{eq13}-\eqref{eq14} satisfy the following estimates under sufficiently regular data
\begin{align}
    & \|\p_t \omega\|_{H^{2m}} \leq Ct^{\alpha_0-1},~~\left\|\p_t^2 \omega \right\|_{ H^{2m}} \leq Ct^{\alpha_0-2},~~a.e.~ t\in(0,T],~~ m=0,1,
\end{align}
which implies
\begin{equation}\label{qqq}
 t \left\|\p_t^2 \omega \right\|_{ H^{2m}} + \left\|\p_t \omega \right\|_{ H^{2m}} \leq Ct^{\alpha_0-1}, ~~a.e.~ t\in (0,T],~~ m=0,1.
\end{equation}

Finally, we impose the condition $\alpha'(0)=0$ throughout the work. It has been demonstrated in \cite[Section 3.1]{ZheWanQiu} that this constraint can be arbitrarily weak. Specifically, for any smooth function $\alpha(t)$, one could construct a sequence of smooth functions $\{\alpha_{\sigma}(t)\}_{\sigma> 0}$ satisfying $\alpha_\sigma'(0)=0$ such that $\max_{t\in [0,T]}|\alpha(t)-\alpha_{\sigma}(t)|\rightarrow 0$ as $\sigma\rightarrow 0$.

\begin{lemma}\label{lem2.3}
 If $\alpha'(0)=0$, we have $|q^{(j)}(t)|\leq C$ for $t\in [0,T]$ and $j=0,1$.
\end{lemma}
\begin{proof} From \eqref{eq10}, we have $|\mathcal{G}(t\varsigma)|\leq C$ such that $|q(t)|\leq C$ for $t\in [0,T]$. To bound $q'(t)$, we differentiate \eqref{eq10} to get
\begin{equation*}
   \begin{split}
       q'(t) = \int_0^1 \frac{ \p_t\mathcal{G}(t\varsigma) \Gamma(1-\alpha(t\varsigma)) + \varsigma\alpha'(t\varsigma)\Gamma'(1-\alpha(t\varsigma))\mathcal{G}(t\varsigma) }{[\Gamma(1-\alpha(t\varsigma))]^2} \frac{(1-\varsigma)^{\alpha_0-1}\varsigma^{-\alpha_0}}{\Gamma(\alpha_0)} \d \varsigma.
   \end{split}
\end{equation*}
Since $\alpha(t)\in(0,1)$, it remains to bound $|\p_t\mathcal{G}(t\varsigma)|$ and $|\Gamma'(1-\alpha(t\varsigma))|$. First, we discuss the estimate of $|\Gamma'(1-\alpha(t\varsigma))|$. By \cite{Askey}, we have
\begin{equation*}
   \begin{split}
       \frac{1}{\Gamma(\jmath)} = \jmath \; e^{\gamma \jmath}  \prod\limits_{n=1}^{\infty} \left( 1+ \frac{\jmath}{n} \right) e^{-\jmath/n}, \quad \gamma\approx 0.5772,
   \end{split}
\end{equation*}
which implies
\begin{equation}\label{eq15}
   \begin{split}
        -\ln \Gamma(\jmath) = \ln\jmath + \gamma \jmath +  \sum\limits_{n=1}^{\infty} \left[  \ln \left(1+\frac{\jmath}{n}\right) - \frac{\jmath}{n} \right].
   \end{split}
\end{equation}
Noting that $\ln \left(1+\frac{\jmath}{n}\right) - \frac{\jmath}{n}=O(\jmath^2/n^2)$ as $n\rightarrow \infty$, thus we can differentiate \eqref{eq15} to obtain that
\begin{equation*}
   \begin{split}
       \frac{\Gamma'(\jmath)}{\Gamma(\jmath)} = -\gamma + \sum\limits_{n=1}^{\infty} \left(  \frac{1}{n} - \frac{1}{n+\jmath} \right) - \frac{1}{\jmath}
       = -\gamma + \sum\limits_{n=0}^{\infty}  \frac{\jmath-1}{(n+1)(n+\jmath)}.
   \end{split}
\end{equation*}
If $\jmath \in (0,1)$, we have $|\Gamma'(\jmath)|\leq C$, which implies that $|\Gamma'(1-\alpha(t\varsigma))|\leq C$. If $\alpha'(0)=0$, we get
\begin{equation*}
   \begin{split}
       \mathcal{G}(t\varsigma) = e^{[\alpha(0)-\alpha(t\varsigma)] \ln(tz)} = e^{-\int_0^{t\varsigma}(t\varsigma-s)\alpha''(s)\d s\ln(t\varsigma)},
   \end{split}
\end{equation*}
thus we have
\begin{equation*}
   \begin{split}
      \p_t \mathcal{G}(t\varsigma) = -\mathcal{G}(t\varsigma)\varsigma \left[ \frac{1}{t\varsigma} \int_0^{t\varsigma}(t\varsigma-s)\alpha''(s)\d s + \int_0^{t\varsigma}\alpha''(s)\d s \ln(t\varsigma) \right].
   \end{split}
\end{equation*}
Then we yield $|\p_t D(t\varsigma)|\leq C$ with $t\in[0,T]$. We finish the proof of the lemma.
\end{proof}

\section{Time-discrete scheme}\label{Sec3}
In this section, our aim is to establish a temporal semi-discrete scheme for the problem \eqref{eq13}-\eqref{eq14} using a quadrature rule to approximate the convolution terms and give a theoretical analysis of the semi-discrete scheme. For simplicity, we will drop the spatial variable $x$ in the function, e.g. we denote $\omega(x,t)$ as $\omega(t)$.

\subsection{Establishment of time-discrete scheme}
Given a positive integer $N$, we discretize the temporal interval $[0,T]$ into $N$-subintervals such that  the time step size $\tau = T/N$ and $t_n=n\tau$, with $n=0,1,\cdots,N$. Next, we consider the transformed equation \eqref{eq13} at the point $t=t_n$, for $ 1\leq n \leq N$,
   \begin{align}\label{eq16}
        \omega(t_n) + (q' * \omega)(t_n) -  \frac{1}{2}\sigma^2c(\beta_{\alpha_0} *  \p_x^2 \omega)(t_n) + \lambda\beta_{\alpha_0} * \omega(t_n) = F(t_n).
   \end{align}
In order to discretize the nonlocal terms in \eqref{eq16}, we introduce the following piecewise linear interpolating to obtain
   \begin{align}
       &  (q' * \omega)(t_n) = \varphi_{n}(\omega) + (\Upsilon_1)^n,  \quad 1\leq n \leq N, \label{eq17} \\
        &  (\beta_{\alpha_0} * \omega)(t_n) = \widetilde{\varphi}_{n}(\omega) + (\Upsilon_2)^n, \quad 1\leq n \leq N, \label{eq18}\\
        &  (\beta_{\alpha_0} * \p_x^2 \omega)(t_n) = \widetilde{\varphi}_{n}(\p_x^2 \omega) + (\Upsilon_3)^n, \quad 1\leq n \leq N, \label{eq19}
   \end{align}
in which,
   \begin{align}
       &  \varphi_{n}(\omega) = \sum\limits_{j=1}^n \varpi_{n,j}\omega(t_j),  \quad  \varpi_{n,j} = \frac{1}{\tau} \int_{t_{n-1}}^{t_n}\int_{t_{j-1}}^{\min(t,t_j)}q'(t-s)\d s\d t, \label{eq20} \\
       &  \widetilde{\varphi}_{n}(\omega) = \sum\limits_{j=1}^n \widetilde{\varpi}_{n,j}\omega(t_j), \quad \widetilde{\varphi}_{n}(\p_x^2 \omega) = \sum\limits_{j=1}^n \widetilde{\varpi}_{n,j}\p_x^2 \omega(t_j),\label{eq21}\\
        &   \widetilde{\varpi}_{n,j} = \frac{1}{\tau} \int_{t_{n-1}}^{t_n}\int_{t_{j-1}}^{\min(t,t_j)} \beta_{\alpha_0}(t-s)\d s \d t,\label{eq22}
   \end{align}
and from \cite{MclTho} it follows that
   \begin{align}
       &  \left| (\Upsilon_1)^n\right| \leq 2\delta_{n-1}\int_0^{\tau}|\p_t \omega(t)|\d t + \tau\sum\limits_{j=2}^{n}\delta_{n-j}\int_{t_{j-1}}^{t_j}\left|\p_{t}^2\omega(t)\right|\d t, \label{eq23} \\
        &  \left| (\Upsilon_2)^n\right| \leq 2\widetilde{\delta}_{n-1}\int_0^{\tau}|\p_t \omega(t)|\d t + \tau\sum\limits_{j=2}^{n}\widetilde{\delta}_{n-j}\int_{t_{j-1}}^{t_j}\left|\p_{t}^2 \omega(t)\right|\d t, \label{eq24}\\
       &  \left| (\Upsilon_3)^n\right| \leq 2\widetilde{\delta}_{n-1}\int_0^{\tau}|\p_t \p_x^2 \omega(t)|\d t + \tau\sum\limits_{j=2}^{n}\widetilde{\delta}_{n-j}\int_{t_{j-1}}^{t_j}\left|\p_{t}^2\p_x^2 \omega(t)\right|\d t, \label{eq25}
   \end{align}
with the notations
\begin{equation}\label{eq26}
   \begin{split}
       \delta_j = \int_{t_{j}}^{t_{j+1}} |q'(t)|\d t, \quad \widetilde{\delta}_j = \int_{t_{j}}^{t_{j+1}} |\beta_{\alpha_0}(t)|\d t.
   \end{split}
\end{equation}
Accordingly, we substitute \eqref{eq17}-\eqref{eq19} into \eqref{eq16} to yield with $\omega^n := \omega(t_n)$ and $F^n:=F(t_n)$,
   \begin{align}
        \omega^n + \sum\limits_{j=1}^n \varpi_{n,j}\omega^j + \lambda\sum\limits_{j=1}^n \widetilde{\varpi}_{n,j}\omega^j &- \frac{1}{2}\sigma^2 \sum\limits_{j=1}^n \widetilde{\varpi}_{n,j}\p_x^2 \omega^j \notag\\
        &= F^n + R^n, \quad 1\leq n \leq N,\label{eq27}
   \end{align}
where $R^n=\frac{1}{2}\sigma^2 (\Upsilon_3)^n-(\Upsilon_1)^n-\lambda(\Upsilon_2)^n$. Following this, by omitting the truncation errors $R^n$ and replacing $\omega^n$ with its numerical approximation $\mathcal{W}^n$, we achieve the following time-discrete scheme
   \begin{align}
       & \mathcal{W}^n + \sum\limits_{j=1}^n \varpi_{n,j}\mathcal{W}^j + \lambda\sum\limits_{j=1}^n \widetilde{\varpi}_{n,j}\mathcal{W}^j -  \frac{1}{2}\sigma^2\sum\limits_{j=1}^n \widetilde{\varpi}_{n,j}\p_x^2 \mathcal{W}^j = F^n,~1\leq n \leq N, \label{eq28} \\
       & \mathcal{W}^0 = 0. \label{eq29}
   \end{align}
Since $\mathcal{W}^n$ is a solution to \eqref{eq28}-\eqref{eq29}, we can directly obtain a numerical solution to \eqref{VtFDEs} through the relation $U^n = \exp\left((\frac{1}{2}-\frac{r}{\sigma^2})x\right)(\mathcal{W}^n + \bar c_t^*)$.
\subsection{Analysis of time-discrete scheme}
Next, we give the stability and convergence of the time-discrete scheme \eqref{eq28}-\eqref{eq29} by means of energy argument. First we prove stability results for the time semi-discrete scheme.
\begin{theorem}\label{thm3.1}
   Let $u(t_n)$ be the solution of \eqref{VtFDEs}, and $U^n$ is the numerical approximation of $u(t_n)$. Assume that Lemma \ref{lem2.3} holds, then 
   \begin{equation*}
       \left\|U^N \right\|_A \leq C(T) \left( \|\bar c_t\| +  \sqrt{ \tau \sum\limits_{n=1}^{N} \|F^n\|^2 } \right),
   \end{equation*}
where the norm $\|U^N\|_A : = \sqrt{\tau\sum\limits_{l=1}^{N}\|U^l\|^2}$, and $\tau$ sufficiently small.
\end{theorem}
\begin{proof}
We take the inner product of \eqref{eq28} with $\tau \mathcal{W}^n$ and sum for $n$ from $1$ to $N$, then
\begin{equation}\label{eq30}
   \begin{split}
        \tau\sum\limits_{n=1}^{N}\|\mathcal{W}^n\|^2 & + \tau\sum\limits_{n=1}^{N}\left( \sum\limits_{j=1}^n \varpi_{n,j} \mathcal{W}^j, \mathcal{W}^n \right) + \lambda\tau\sum\limits_{n=1}^{N} \sum\limits_{j=1}^n \widetilde{\varpi}_{n,j}(  \mathcal{W}^j, \mathcal{W}^n)\\
         & +   \frac{1}{2}\sigma^2\tau\sum\limits_{n=1}^{N}\sum\limits_{j=1}^n \widetilde{\varpi}_{n,j} ( \nabla \mathcal{W}^j, \nabla \mathcal{W}^n) =  \tau\sum\limits_{n=1}^{N}( F^n, \mathcal{W}^n).
   \end{split}
\end{equation}
Since the kernel $\beta_{\alpha_0}(t)$ is positive definite with $\mathcal{W}^0=0$, we have
\begin{equation}\label{eq31}   \tau\sum\limits_{n=1}^{N}\sum\limits_{j=1}^n \widetilde{\varpi}_{n,j} (\nabla \mathcal{W}^j, \nabla \mathcal{W}^n) \geq 0,
\end{equation}
and
\begin{equation}\label{eq32}
\lambda\tau\sum\limits_{n=1}^{N} \sum\limits_{j=1}^n \widetilde{\varpi}_{n,j}( \mathcal{W}^j, \mathcal{W}^n)\geq 0,
\end{equation}
see \cite{MclTho}. Then, we utilize the Cauchy-Schwarz inequality, \eqref{eq31}, \eqref{eq32} and Young's inequality to yield
\begin{equation*}
   \begin{split}
       \frac{1}{2} \tau\sum\limits_{n=1}^{N}\|\mathcal{W}^n\|^2 & \leq   \tau\sum\limits_{n=1}^{N}\left\| \sum\limits_{j=1}^n \varpi_{n,j} \mathcal{W}^j\right\|^2 + \tau\sum\limits_{n=1}^{N}\left\| F^n\right\|^2 \\
       & \leq \tau\sum\limits_{n=1}^{N}\left( \sum\limits_{j=1}^n |\varpi_{n,j}| \left\|\mathcal{W}^j\right\| \right)^2 + \tau\sum\limits_{n=1}^{N}\left\| F^n\right\|^2.
   \end{split}
\end{equation*}
By defining a new norm
\begin{equation}\label{eq33}
   \begin{split}
       \|\mathcal{W}^n\|_A : = \sqrt{\tau\sum\limits_{l=1}^{n}\|\mathcal{W}^l\|^2}, \quad 1\leq n \leq N,
   \end{split}
\end{equation}
then we have
\begin{equation*}
   \begin{split}
        \|\mathcal{W}^N\|_A^2 & \leq 2\tau\sum\limits_{n=1}^{N}\left( \sum\limits_{j=1}^n |\varpi_{n,j}| \left\|\mathcal{W}^j\right\| \right)^2 + 2\tau\sum\limits_{n=1}^{N}\left\| F^n\right\|^2 \\
        & \leq C\tau\sum\limits_{n=1}^{N}  \|\mathcal{W}^n\|_A^2 + 2\tau\sum\limits_{n=1}^{N}\left\| F^n\right\|^2,
   \end{split}
\end{equation*}
where we use the fact that $|\varpi_{n,j}|\leq C\tau$, see \eqref{eq20} and Lemma \ref{lem2.3}. For the formula above, the discrete Gr\"{o}nwall's lemma gives by taking $\tau\leq \frac{1}{2C}$,
\begin{equation*}
   \begin{split}
        \left\|\mathcal{W}^N\right\|_A^2 \leq C(T) \sum\limits_{n=1}^{N} \tau  \left\| F^n\right\|^2, \quad N \geq 1.
   \end{split}
\end{equation*}
Let $U^n$ be the numerical approximation of $u(t)$ of the problem \eqref{VtFDEs} at the mesh point $t_n$, respectively.
Based on $U^n = \exp\left((\frac{1}{2}-\frac{r}{\sigma^2})x\right)(\mathcal{W}^n + \bar c_t^*)$, it can be obtained
\begin{equation*}
       \left\|U^N \right\|_A \leq C(T) \left( \|\bar c_t\| +  \sqrt{ \tau \sum\limits_{n=1}^{N} \|F^n\|^2 } \right).
   \end{equation*}
The proof is completed.
\end{proof}

\vskip 1mm
We now consider the convergence of the time-discrete scheme \eqref{eq28}-\eqref{eq29}. Denote
\begin{equation*}
   \Theta^n = \omega(t_n) - \mathcal{W}^n, \quad 0\leq n \leq N.
\end{equation*}
Then, we subtract \eqref{eq28} from \eqref{eq27} to yield the following error equation
   \begin{align}\label{eq34}
        \Theta^n + \sum\limits_{j=1}^n \varpi_{n,j}\Theta^j + \lambda\sum\limits_{j=1}^n \widetilde{\varpi}_{n,j}\Theta^j -  \frac{1}{2}\sigma^2\sum\limits_{j=1}^n \widetilde{\varpi}_{n,j}\p_x^2 \Theta^j = R^n, \quad 1\leq n \leq N,
   \end{align}
with $\Theta^0 = 0$. Subsequently, we establish the following convergence result.

\vskip 1mm
\begin{theorem}\label{them4} Let $u(t_n)$ be the solution of \eqref{VtFDEs} and $U^n$ be the numerical approximation of $u(t_n)$. Then, based on Lemma \ref{lem2.3}, it holds that
\begin{equation*}
    \|u(t_N)-U^N\|_A \leq C(T)  \tau^{\frac{1}{2}+\frac{3}{2}\alpha_0}.
\end{equation*}
\end{theorem}
\begin{proof} Analysing similarly to Theorem \ref{thm3.1}, we use \eqref{eq34} to derive
   \begin{equation*}
      \begin{split}
       \left\|\Theta^N \right\|_A & \leq C(T) \left( \|\Theta^0\| +  \sqrt{ \tau \sum\limits_{n=1}^{N} \|R^n\|^2 } \right) \\
       & \leq C(T) \left( \sqrt{ \tau \sum\limits_{n=1}^{N} \|(\Upsilon_1)^n\|^2 } +  \sqrt{ \tau \sum\limits_{n=1}^{N} \|(\Upsilon_2)^n\|^2 } +  \sqrt{ \tau \sum\limits_{n=1}^{N} \|(\Upsilon_3)^n\|^2 }\right).
       \end{split}
   \end{equation*}
We analyse each term on the right-hand side of the equation. By \eqref{qqq}, Lemma \ref{lem2.3} and \eqref{eq23} arrives at
   \begin{equation*}
      \begin{split}
         \|(\Upsilon_1)^n\| \leq C\tau \int_0^{\tau} t^{\alpha_0 -1}\d t + C\tau^2 \int_{\tau}^{t_n} t^{\alpha_0-2}\d t \leq C\tau^{\alpha_0 +1},
       \end{split}
   \end{equation*}
where we employ $\delta_n \leq C\tau$. It concludes that
   \begin{equation*}
      \begin{split}
            \tau \sum\limits_{n=1}^{N} \|(\Upsilon_1)^n\|^2  \leq  C\tau^{2(\alpha_0 +1)}.
       \end{split}
   \end{equation*}
Utilizing \eqref{eq26}, we have $\widetilde{\delta}_n = \frac{t_{n+1}^{\alpha_0}-t_{n}^{\alpha_0}}{\Gamma(\alpha_0+1)} \leq C\tau^{\alpha_0}$, thus combined with \eqref{eq24} yields
   \begin{equation*}
      \begin{split}
            \|(\Upsilon_2)^n\| \leq C\tau^{\alpha_0} \int_0^{\tau} t^{\alpha_0 -1}\d t + C\tau^{1+\alpha_0} \int_{\tau}^{t_n} t^{\alpha_0-2}\d t \leq C\tau^{2\alpha_0}.
       \end{split}
   \end{equation*}
In addition, we have
   \begin{equation*}
      \begin{split}
            \tau \sum\limits_{n=1}^{N} \|(\Upsilon_2)^n\| &\leq  C \tau \sum\limits_{n=1}^{N} \frac{t_{n}^{\alpha_0}-t_{n-1}^{\alpha_0}}{\Gamma(\alpha_0+1)} \int_0^{\tau} t^{\alpha_0 -1}\d t + C\tau^2 \sum\limits_{n=2}^{N} \sum\limits_{j=2}^{n} \widetilde{\delta}_{n-j} \int_{t_{j-1}}^{t_j} t^{\alpha_0-2}\d t  \\
            & \leq C\tau^{\alpha_0+1} +  C\tau^2 \sum\limits_{j=2}^{N} \int_{t_{j-1}}^{t_j} t^{\alpha_0-2}\d t \left( \sum\limits_{l=0}^{N-j} \widetilde{\delta}_{l} \right) \leq C\tau^{\alpha_0+1}.
       \end{split}
   \end{equation*}
Therefore, the above two formulas deduce
   \begin{equation*}
      \begin{split}
            \tau \sum\limits_{n=1}^{N} \|(\Upsilon_2)^n\|^2  \leq  C\tau^{3\alpha_0+1}.
       \end{split}
   \end{equation*}
  Similarly, we get
  \begin{equation*}
      \begin{split}
            \tau \sum\limits_{n=1}^{N} \|(\Upsilon_3)^n\|^2  \leq  C\tau^{3\alpha_0+1},
       \end{split}
   \end{equation*}
which finishes the proof of Theorem \ref{them4}.
\end{proof}

\section{Fully discrete finite element scheme}\label{Sec4}
In this section, we give a fully discrete finite element scheme based on a time semi-discrete scheme and derive its convergence result.

Define a quasi-uniform partition of $[\bar a,\bar b]$ with mesh diameter $h$ and let $S_h$ be the space of continuous and piecewise linear functions on $[\bar a,\bar b]$ with respect to the partition. Let $I$ be the identity operator. The Ritz projection $\mathrm {\Pi}_h:H_0^1(\bar a,\bar b)\to S_h$ defined by $(\nabla(\psi-\mathrm {\Pi}_h \psi),\nabla \varrho)=0$ for any $\varrho\in S_h$ has the approximation property \cite{Tho}
\begin{equation}
 \begin{array}{lll}\label{eq48}
       \|(I-\mathrm {\Pi}_h )\psi\|\leq Ch^2\|\psi\|_{H^2}, \quad \forall \psi\in H^2(\bar a,\bar b)\cap H_0^1(\bar a,\bar b).
   \end{array}
   \end{equation}

Multiply equation \eqref{eq27} by $\Lambda\in H_0^1[\bar a,\bar b])$ on $[\bar a,\bar b]$ to obtain the weak formulation for any $\Lambda\in H_0^1([\bar a,\bar b])$ and $n=1,2,\dots,N$
 \begin{align}
        (\omega^n,\Lambda) + (\sum\limits_{j=1}^n \varpi_{n,j}\omega^j,\Lambda) &+ \lambda(\sum\limits_{j=1}^n \widetilde{\varpi}_{n,j}\omega^j,\Lambda) +\frac{1}{2}\sigma^2 (\sum\limits_{j=1}^n \widetilde{\varpi}_{n,j}\nabla \omega^j,\nabla\Lambda) \notag\\
        &= (F^n,\Lambda) + (R^n,\Lambda), \quad 1\leq n \leq N.\label{eq41}
   \end{align}
Drop the local truncation error terms in \eqref{eq41} to obtain a fully-discrete finite element scheme: find $\mathcal{W}_h^n\in S_h$ for $1\leq n\leq N$ such that
\begin{align}
       (\mathcal{W}_h^n,\Lambda) + (\sum\limits_{j=1}^n \varpi_{n,j}\mathcal{W}^j,\Lambda) &+ \lambda(\sum\limits_{j=1}^n \widetilde{\varpi}_{n,j}\mathcal{W}_h^j,\Lambda) +\frac{1}{2}\sigma^2 (\sum\limits_{j=1}^n \widetilde{\varpi}_{n,j}\nabla \mathcal{W}_h^j,\nabla\Lambda) \notag\\
        &= (F^n,\Lambda), \quad \forall\Lambda\in S_h.\label{eq42}
   \end{align}
After solving the numerical solution $\mathcal{W}_h^n$ in \eqref{eq42}, we further define the numerical solution $U_h^n$ of equation \eqref{VtFDEs} as follows
\begin{equation*}
U_h^n =  \exp\left((\frac{1}{2}-\frac{r}{\sigma^2})x\right)(\mathcal{W}_h^n + \bar c_t^*).
  \end{equation*}

\subsection{Error estimate of fully discrete scheme}
We now estimate the error of the full discrete scheme \eqref{eq42} in the following theorem.
\begin{theorem}\label{thm6}
  Assume that $\omega$ satisfies the regularity assumption in \eqref{qqq} and $\mathcal{W}_h^n$ be the numerical approximation of $\omega$. Then, the following convergence result holds with $T<\infty$ and $\tau$ sufficiently small
\begin{equation}\label{eqthm6}
\sqrt{\tau\sum\limits_{n=1}^{N}\|\omega^n-\mathcal{W}_h^n}\|^2
       \leq C(T)(\tau^{\frac{1}{2}+\frac{3}{2}\alpha_0}+h^2).
   \end{equation}
\end{theorem}
\begin{proof}
Set $\omega^n - \mathcal{W}_h^n = \xi^n + \eta^n$ with $\eta^n := \omega^n-\rm \Pi_h\omega^n$ and $\xi^n :=\rm \Pi_h\omega^n - \mathcal{W}_h^n$. Then we subtract \eqref{eq42} from \eqref{eq41} and select $\Lambda=\xi^n$ to obtain
\begin{align}
        (\xi^n,\xi^n) &+ (\sum\limits_{j=1}^n \varpi_{n,j}\xi^j,\xi^n) + \lambda(\sum\limits_{j=1}^n \widetilde{\varpi}_{n,j}\xi^j,\xi^n) +\frac{1}{2}\sigma^2 (\sum\limits_{j=1}^n \widetilde{\varpi}_{n,j}\nabla \xi^j,\nabla\xi^n) \notag\\
        &= (R^n,\xi^n) - (\eta^n,\xi^n) - (\sum\limits_{j=1}^n \varpi_{n,j}\eta^j,\xi^n) - \lambda(\sum\limits_{j=1}^n \widetilde{\varpi}_{n,j}\eta^j,\xi^n). \label{eq43}
   \end{align}
   We multiply the equation by $\tau$ and sum $n$ from $1$ to $N$ to get
   \begin{align}
        \tau\sum\limits_{n=1}^{N}(\xi^n,\xi^n) &+ \tau\sum\limits_{n=1}^{N}(\sum\limits_{j=1}^n \varpi_{n,j}\xi^j,\xi^n) + \lambda\tau\sum\limits_{n=1}^{N}(\sum\limits_{j=1}^n \widetilde{\varpi}_{n,j}\xi^j,\xi^n)\notag\\
        &+\frac{1}{2}\sigma^2 \tau\sum\limits_{n=1}^{N}(\sum\limits_{j=1}^n \widetilde{\varpi}_{n,j}\nabla \xi^j,\nabla\xi^n) = \tau\sum\limits_{n=1}^{N}(R^n,\xi^n)- \tau\sum\limits_{n=1}^{N}(\eta^n,\xi^n)  \notag\\
        &- \tau\sum\limits_{n=1}^{N}(\sum\limits_{j=1}^n \varpi_{n,j}\eta^j,\xi^n) - \lambda\tau\sum\limits_{n=1}^{N}(\sum\limits_{j=1}^n \widetilde{\varpi}_{n,j}\eta^j,\xi^n). \label{eq44}
   \end{align}
Since the kernel $\beta_{\alpha_0}(t)$ is positive definite with $\xi^0=0$, we have
\begin{equation}\label{eq45}
    \lambda\tau\sum\limits_{n=1}^{N}(\sum\limits_{j=1}^n \widetilde{\varpi}_{n,j}\xi^j,\xi^n) \geq 0,
\end{equation}
and
\begin{equation}\label{eq46}
\frac{1}{2}\sigma^2 \tau\sum\limits_{n=1}^{N}(\sum\limits_{j=1}^n \widetilde{\varpi}_{n,j}\nabla \xi^j,\nabla\xi^n)\geq 0,
\end{equation}
see \cite{MclTho}. Further, we utilize the Cauchy-Schwarz inequality, \eqref{eq45}, \eqref{eq46} and $|ab|\leq \frac{1}{8}a^2+2b^2$ to obtain
\begin{equation*}
   \begin{split}
       \frac{3}{8} \tau\sum\limits_{n=1}^{N}\|\xi^n\|^2  \leq &  2\tau\sum\limits_{n=1}^{N}\left\| \sum\limits_{j=1}^n \varpi_{n,j} \xi^j\right\|^2 + 2\tau\sum\limits_{n=1}^{N}\left\| R^n\right\|^2+ 2\tau\sum\limits_{n=1}^{N}\left\| \eta^n\right\|^2 \\
       & +2\tau\sum\limits_{n=1}^{N}\left\| \sum\limits_{j=1}^n \varpi_{n,j} \eta^j\right\|^2 + 2\tau\lambda\sum\limits_{n=1}^{N}\left\| \sum\limits_{j=1}^n \widetilde{\varpi}_{n,j} \eta^j\right\|^2\\
       \leq &  2\tau\sum\limits_{n=1}^{N}\left( \sum\limits_{j=1}^n |\varpi_{n,j}| \|\xi^j\|\right)^2 + 2\tau\sum\limits_{n=1}^{N}\left\| R^n\right\|^2+ 2\tau\sum\limits_{n=1}^{N}\left\| \eta^n\right\|^2 \\
       & +2\tau\sum\limits_{n=1}^{N}\left(\sum\limits_{j=1}^n |\varpi_{n,j}| \|\eta^j\|\right)^2 + 2\tau\lambda\sum\limits_{n=1}^{N}\left(\sum\limits_{j=1}^n |\widetilde{\varpi}_{n,j}| \|\eta^j\|\right)^2.
   \end{split}
\end{equation*}
By defining a new norm
\begin{equation}\label{eq47}
   \begin{split}
       \|\xi^n\|_B : = \sqrt{\tau\sum\limits_{l=1}^{n}\|\xi^l\|^2}, \quad 1\leq n \leq N,
   \end{split}
\end{equation}
then we have
\begin{equation*}
   \begin{split}
        \|\xi^N\|_B^2  \leq & C\tau\sum\limits_{n=1}^{N}  \|\xi^n\|_B^2 + \frac{16}{3}\tau\sum\limits_{n=1}^{N}\left\| R^n\right\|^2+ \frac{16}{3}\tau\sum\limits_{n=1}^{N}\left\| \eta^n\right\|^2\\
        & +\frac{16}{3}\tau\sum\limits_{n=1}^{N}\left(\sum\limits_{j=1}^n |\varpi_{n,j}| \|\eta^j\|\right)^2 + \frac{16}{3}\tau\lambda\sum\limits_{n=1}^{N}\left(\sum\limits_{j=1}^n |\widetilde{\varpi}_{n,j}| \|\eta^j\|\right)^2.
   \end{split}
\end{equation*}
We utilize $|\varpi_{n,j}|\leq C\tau$, $|\widetilde{\varpi}_{n,j}|\leq C\tau$ and $\|\eta^n\|\leq Ch^2$ to obtain
\begin{equation*}
   \begin{split}
        \|\xi^N\|_B^2  \leq & C\tau\sum\limits_{n=1}^{N}  \|\xi^n\|_B^2 + \frac{16}{3}\tau\sum\limits_{n=1}^{N}\left\| R^n\right\|^2+ C\tau\sum\limits_{n=1}^{N}\left\| \eta^n\right\|^2+C\tau\sum\limits_{n=1}^{N}h^4.
   \end{split}
\end{equation*}
For the formula above, the discrete Gr\"{o}nwall's lemma gives by taking $\tau\leq \frac{1}{2C}$,
\begin{equation*}
   \begin{split}
        \|\xi^N\|_B^2 &\leq C(T) \tau\sum\limits_{n=1}^{N} (\left\| \eta^n\right\|^2+\left\| R^n\right\|^2+h^4)\\
        &\leq (C(T))^2(\tau^{\frac{1}{2}+\frac{3}{2}\alpha_0}+h^2)^2.
   \end{split}
\end{equation*}
We combine this with $\|\omega^N - \mathcal{W}_h^N
\|\leq Ch^2$ to get
\begin{equation*}
   \begin{split}
       \|\omega^N -\mathcal{W}_h^N\|_B^2 \leq C(T)^2(\tau^{\frac{1}{2}+\frac{3}{2}\alpha_0}+h^2)^2.
   \end{split}
\end{equation*}
The proof is completed by using \eqref{eq47}.
\end{proof}

\section{Numerical experiments} \label{Sec5}
In this section, we provide some numerical examples to verify our theoretical results.  {\color{black}  A uniform mesh partitioned by the nodal points
$x_i=ih$ where $0\leq i\leq M$ and $h=\frac{\bar a-\bar b}{M}$ is used in examples. As the exact solutions are in general not available, we measure the $ L^2$ error following the two-mesh principle \cite[Page 107]{Far}}
{\color{black}{\begin{align*}
  & \text{Error}_\tau(N,M)= \sqrt{\tau\sum_{n=1}^{N}h\sum\limits_{j=1}^{M-1} \left|U_j^{2n}(2N,M)-U_j^n(N,M)\right|^2},\\
  &
  \text{Error}_h(N,M)=   \sqrt{\tau\sum_{n=1}^{N}h\sum_{j=1}^{M-1}\left|U_{2j}^{n}(N,2M)-U_j^n(N,M)\right|^2},
\end{align*}}}
and we obtain the temporal and spatial convergence rates by
\begin{align*}
  \text{Order}_\tau = \log_{2}\left(\frac{\text{Error}_\tau(N,M)}{\text{Error}_\tau(2N,M)}\right), \quad \text{Order}_h = \log_{2}\left(\frac{\text{Error}_h(N,M)}{\text{Error}_h(N,2M)}\right).
\end{align*}

\vskip 1mm
\noindent
{\bf Example 1.} Let $(\bar a,\bar b) = (0,1)$, $T=1$, $\sigma=0.45$, $r=0.03$, $u_0(x)=\sin(\pi x)$, the forcing term $f=(\frac{r}{\sigma^4}-\frac{1}{2\sigma^2})\pi\cos(\pi x)$, and $\alpha(t)=\alpha_0-\frac{1}{11}t$ $(\alpha'(0)\neq 0)$, where $\alpha_0 \in [0.1, 1)$. First, with different $\alpha_0$, the $L^2$ errors and  temporal convergence orders are shown in Table \ref{tb1} when $M = 32$, which demonstrate that the method achieves $\frac{1}{2}+\frac{3}{2}\alpha_0$ order convergence in time, which is consistent with the theoretical analysis. In Table \ref{tb2}, we explore the $ L^2$ error and the order of spatial convergence of the finite element scheme when $N = 32$ for different $\alpha_0$. As $M$ increases, the second-order convergence rate in space is observed from Table \ref{tb2}.
\begin{table}
    \centering
    \caption{$ L^2$-norm errors and temporal convergence orders when $M=32$.} \label{tb1}
    \begin{tabular}{cccccccccccc}
      \toprule
     & \multicolumn{2}{c}{$\alpha_0=0.1$} & &\multicolumn{2}{c}{$\alpha_0=0.4$} \\
     \cmidrule{2-3}  \cmidrule{5-6}
      $N$ & $\text{Error}_\tau(N,M)$ & $\text{Order}_\tau$ & $N$ & $\text{Error}_\tau(N,M)$ &
      $\text{Order}_\tau$   \\
      \midrule
       16 &  $1.8337 \times 10^{-3}$   &  *       & 16  &  $2.0227 \times 10^{-3}$  &  *     \\
       32  &  $1.1920 \times 10^{-3}$   &  0.62   & 32   &  $9.1176 \times 10^{-4}$  &  1.15      \\
       64  &  $7.7044 \times 10^{-4}$   &  0.63   & 64   &  $4.0515 \times 10^{-4}$  &  1.17    \\
       128 &  $4.9478 \times 10^{-4}$   &  0.64   & 128  &  $1.7783 \times 10^{-4}$  &  1.19    \\
     Theory &   & 0.65      &  &   &  1.10     \\
      \midrule
     & \multicolumn{2}{c}{$\alpha_0=0.7$} & &\multicolumn{2}{c}{$\alpha_0=0.9$} \\
     \cmidrule{2-3}  \cmidrule{5-6}
      $N$ & $\text{Error}_\tau(N,M)$ & $\text{Order}_\tau$  & $N$ & $\text{Error}_\tau(N,M)$ & $\text{Order}_\tau$   \\
      \midrule
       16  &  $4.7357 \times 10^{-4}$   &  *     & 16   &  $1.3701 \times 10^{-4}$  &  *    \\
       32   &  $1.5230 \times 10^{-4}$   &  1.64     & 32   &  $3.6889 \times 10^{-5}$  &  1.89     \\
       64   &  $4.8774 \times 10^{-5}$   &  1.64   & 64   &  $9.8434 \times 10^{-6}$  &  1.91    \\
       128  &  $1.5554 \times 10^{-5}$   &  1.65   & 128  &  $2.5712 \times 10^{-6}$  &  1.94    \\
       Theory &   & 1.55      &  &   &   1.85     \\
      \bottomrule
    \end{tabular}
\end{table}

\begin{table}
    \centering
    \caption{$L^2$-norm errors and spatial convergence orders when $N=32$.} \label{tb2}
    \begin{tabular}{cccccccccccc}
      \toprule
     & \multicolumn{2}{c}{$\alpha_0=0.1$} & &\multicolumn{2}{c}{$\alpha_0=0.4$} \\
     \cmidrule{2-3}  \cmidrule{5-6}
      $M$ & $\text{Error}_h(N,M)$ & $\text{Order}_h$ & $M$ & $\text{Error}_h(N,M)$ &
      $\text{Order}_h$   \\
      \midrule
       32  &  $1.0586 \times 10^{-4}$   &  *       & 32   &  $9.2023 \times 10^{-5}$   &  *      \\
       64  &  $2.6466 \times 10^{-5}$   &  2.00    & 64   &  $2.3005 \times 10^{-5}$   &  2.00    \\
       128 &  $6.6166 \times 10^{-6}$   &  2.00    & 128  &  $5.7513 \times 10^{-6}$   &  2.00    \\
       256 &  $1.6542 \times 10^{-6}$   &  2.00    & 256  &  $1.4378 \times 10^{-6}$   &  2.00    \\
    Theory &   & 2.00     &  &   &  2.00     \\
     \midrule
     & \multicolumn{2}{c}{$\alpha_0=0.7$} & &\multicolumn{2}{c}{$\alpha_0=0.9$} \\
     \cmidrule{2-3}  \cmidrule{5-6}
      $M$ & $\text{Error}_h(N,M)$ & $\text{Order}_h$  & $M$ & $\text{Error}_h(N,M)$ & $\text{Order}_h$   \\
      \midrule
       32  &  $7.7674 \times 10^{-5}$    &  *      & 32   &  $6.5094 \times 10^{-5}$   &  *      \\
       64  &  $1.9417 \times 10^{-5}$    &  2.00   & 64   &  $1.6271 \times 10^{-5}$   &  2.00   \\
       128 &  $4.8540 \times 10^{-6}$    &  2.00   & 128  &  $4.0675 \times 10^{-6}$   &  2.00   \\
       256 &  $1.2135 \times 10^{-6}$    &  2.00   & 256  &  $1.0169 \times 10^{-6}$   &  2.00    \\
       Theory &   & 2.00     &  &   &  2.00     \\
      \bottomrule
    \end{tabular}
\end{table}
\vskip 2mm

\vskip 1mm
\noindent{\bf Example 2.} Let $(\bar a,\bar b)=(0,1)$, $T=1$, $u_0(x)=\sin(\pi x)$, the forcing term $f=0$, $\sigma = 0.5$, $r=0.25$ and a nonlinear exponent $\alpha(t)=\alpha_0-\frac{1}{11}t^2$ ($\alpha'(0)= 0$), where $\alpha_0 \in [0.1, 1)$. We get the same conclusions as Example 1 from Tables \ref{tb3}--\ref{tb4}.

\begin{table}
    \center \footnotesize
    \caption{$ L^2$-norm errors and temporal convergence orders when $M=32$.} \label{tb3}
    \begin{tabular}{cccccccccccc}
      \toprule
     & \multicolumn{2}{c}{$\alpha_0=0.1$} & &\multicolumn{2}{c}{$\alpha_0=0.4$} \\
     \cmidrule{2-3}  \cmidrule{5-6}
      $N$ & $\text{Error}_\tau(N,M)$ & $\text{Order}_\tau$ & $N$ & $\text{Error}_\tau(N,M)$ &
      $\text{Order}_\tau$   \\
      \midrule
       32  &  $2.1244 \times 10^{-3}$   &  *      & 32   &  $1.9582 \times 10^{-4}$  &  *      \\
       64  &  $1.3822 \times 10^{-3}$   &  0.62   & 64   &  $1.0376 \times 10^{-4}$  &  0.92    \\
       128 &  $8.9649 \times 10^{-4}$   &  0.62   & 128  &  $5.2783 \times 10^{-5}$  &  0.98    \\
       256 &  $5.7974 \times 10^{-4}$   &  0.63   & 256  &  $2.5898 \times 10^{-5}$  &  1.03    \\
     Theory &   & 0.65      &  &   &  1.10     \\
      \midrule
     & \multicolumn{2}{c}{$\alpha_0=0.7$} & &\multicolumn{2}{c}{$\alpha_0=0.9$} \\
     \cmidrule{2-3}  \cmidrule{5-6}
      $N$ & $\text{Error}_\tau(N,M)$ & $\text{Order}_\tau$  & $N$ & $\text{Error}_\tau(N,M)$ & $\text{Order}_\tau$   \\
      \midrule
       32   &  $3.6093 \times 10^{-4}$   &  *      & 32   &  $1.0412 \times 10^{-4}$  &  *      \\
       64   &  $1.2192 \times 10^{-4}$   &  1.57   & 64   &  $2.9759 \times 10^{-5}$  &  1.81   \\
       128  &  $4.1391 \times 10^{-5}$   &  1.56   & 128  &  $8.4148 \times 10^{-6}$  &  1.82    \\
       256  &  $1.4119 \times 10^{-5}$   &  1.55   & 256  &  $2.3241 \times 10^{-6}$  &  1.86    \\
       Theory &   & 1.55      &  &   &   1.85     \\
      \bottomrule
    \end{tabular}
\end{table}

\newpage
\begin{table}
    \center \footnotesize
    \caption{$ L^2$-norm errors and spatial convergence orders when $N=32$.} \label{tb4}
    \begin{tabular}{cccccccccccc}
      \toprule
     & \multicolumn{2}{c}{$\alpha_0=0.1$} & &\multicolumn{2}{c}{$\alpha_0=0.4$} \\
     \cmidrule{2-3}  \cmidrule{5-6}
      $M$ & $\text{Error}_h(N,M)$ & $\text{Order}_h$ & $M$ & $\text{Error}_h(N,M)$ &
      $\text{Order}_h$   \\
      \midrule
       32  &  $2.4148 \times 10^{-4}$   &  *       & 32   &  $2.3292 \times 10^{-4}$   &  *      \\
       64  &  $6.1457 \times 10^{-5}$   &  1.97    & 64   &  $5.9371 \times 10^{-5}$   &  1.97    \\
       128 &  $1.5505 \times 10^{-5}$   &  1.99    & 128  &  $1.4990 \times 10^{-5}$   &  1.99    \\
       256 &  $3.8940 \times 10^{-6}$   &  1.99    & 256  &  $3.7664 \times 10^{-6}$   &  1.99    \\
    Theory &   & 2.00     &  &   &  2.00     \\
     \midrule
     & \multicolumn{2}{c}{$\alpha_0=0.7$} & &\multicolumn{2}{c}{$\alpha_0=0.9$} \\
     \cmidrule{2-3}  \cmidrule{5-6}
      $M$ & $\text{Error}_h(N,M)$ & $\text{Order}_h$  & $M$ & $\text{Error}_h(N,M)$ & $\text{Order}_h$   \\
      \midrule
       32  &  $2.2778 \times 10^{-4}$    &  *      & 32   &  $2.2503 \times 10^{-4}$   &  *      \\
       64  &  $5.8120 \times 10^{-5}$    &  1.97   & 64   &  $5.7465 \times 10^{-5}$   &  1.97   \\
       128 &  $1.4683 \times 10^{-5}$    &  1.98   & 128  &  $1.4525 \times 10^{-5}$   &  1.98  \\
       256 &  $3.6902 \times 10^{-6}$    &  1.99   & 256  &  $3.6517 \times 10^{-6}$   &  1.99    \\
       Theory &   & 2.00     &  &   &  2.00     \\
      \bottomrule
    \end{tabular}
\end{table}

\vskip 2mm
{\color{black}{
\noindent{\bf Example 3.} Let $(\bar a,\bar b) = (0,1)$, $T=1$, $\sigma=0.4$, $r=0.1$, $u_0(x)=\sin(\pi x)$, the forcing term $f=(\frac{r}{\sigma^4}-\frac{1}{2\sigma^2})\pi\cos(\pi x)$, and $\alpha(t)=\alpha_0+\frac{1}{11}t^3$ $(\alpha'(0)= 0)$, where $\alpha_0 \in (0, 0.9]$. Different from the first two examples, the  $\alpha(t)$ is an increasing function. We get the same conclusions as Example 1 from Tables \ref{tb5}--\ref{tb6}, which indicates that the monotonicity of $\alpha(t)$ does not affect the numerical accuracy.

\begin{table}
    \centering
    \caption{$ L^2$-norm errors and temporal convergence orders when $M=32$.} \label{tb5}
    \begin{tabular}{cccccccccccc}
      \toprule
     & \multicolumn{2}{c}{$\alpha_0=0.1$} & &\multicolumn{2}{c}{$\alpha_0=0.4$} \\
     \cmidrule{2-3}  \cmidrule{5-6}
      $N$ &  $\text{Error}_\tau(N,M)$  & $\text{Order}_\tau$& $N$ & $\text{Error}_\tau(N,M)$  &
      $\text{Order}_\tau$ \\
      \midrule
       4  &  $3.6526 \times 10^{-3}$   &  *      & 4   &  $7.3880 \times 10^{-3}$  &  *      \\
       8  &  $2.3790 \times 10^{-3}$   &  0.62   & 8   &  $3.4375 \times 10^{-3}$  &  1.10    \\
       16 &  $1.5270 \times 10^{-3}$   &  0.64   & 16  &  $1.5660 \times 10^{-3}$  &  1.13    \\
       32 &  $9.7177 \times 10^{-4}$   &  0.65   & 32  &  $7.0100 \times 10^{-4}$  &  1.16    \\
     Theory &   & 0.65      &  &   &  1.10     \\
      \midrule
     & \multicolumn{2}{c}{$\alpha_0=0.7$} & &\multicolumn{2}{c}{$\alpha_0=0.9$} \\
     \cmidrule{2-3}  \cmidrule{5-6}
      $N$ &  $\text{Error}_\tau(N,M)$ & $\text{Order}_\tau$ & $N$ &  $\text{Error}_\tau(N,M)$  & $\text{Order}_\tau$   \\
      \midrule
       4   &  $3.3174 \times 10^{-3}$   &  *      & 4   &  $1.3359 \times 10^{-3}$  &  *      \\
       8   &  $1.0746 \times 10^{-3}$   &  1.63   & 8   &  $3.5407 \times 10^{-4}$  &  1.92   \\
       16  &  $3.4713 \times 10^{-4}$   &  1.63   & 16  &  $9.5064 \times 10^{-5}$  &  1.90    \\
       32  &  $1.1170 \times 10^{-4}$   &  1.64   & 32  &  $2.5463 \times 10^{-5}$  &  1.90    \\
       Theory &   & 1.55      &  &   &   1.85     \\
      \bottomrule
    \end{tabular}
\end{table}

\newpage
\begin{table}
    \centering
    \caption{$ L^2$-norm errors and spatial convergence orders when $N=32$.} \label{tb6}
    \begin{tabular}{cccccccccccc}
      \toprule
     & \multicolumn{2}{c}{$\alpha_0=0.1$} & &\multicolumn{2}{c}{$\alpha_0=0.4$} \\
     \cmidrule{2-3}  \cmidrule{5-6}
      $M$ & $\text{Error}_h(N,M)$ &  $\text{Order}_h$ & $M$ & $\text{Error}_h(N,M)$ &
       $\text{Order}_h$   \\
      \midrule
       4  &  $5.3523 \times 10^{-3}$   &  *       & 4   &  $4.6527 \times 10^{-3}$   &  *      \\
       8  &  $1.3414 \times 10^{-3}$   &  2.00    & 8   &  $1.1635 \times 10^{-3}$   &  2.00    \\
       16 &  $3.3542 \times 10^{-4}$   &  2.00    & 16  &  $2.9076 \times 10^{-4}$   &  2.00    \\
       32 &  $8.3856 \times 10^{-5}$   &  2.00    & 32  &  $7.2682 \times 10^{-5}$   &  2.00    \\
    Theory &   & 2.00     &  &   &  2.00     \\
     \midrule
     & \multicolumn{2}{c}{$\alpha_0=0.7$} & &\multicolumn{2}{c}{$\alpha_0=0.9$} \\
     \cmidrule{2-3}  \cmidrule{5-6}
      $M$ & $\text{Error}_h(N,M)$ &  $\text{Order}_h$ & $M$ & $\text{Error}_h(N,M)$ &  $\text{Order}_h$   \\
      \midrule
       4  &  $3.8724 \times 10^{-3}$    &  *      & 4   &  $3.1934 \times 10^{-3}$   &  *      \\
       8  &  $9.6488 \times 10^{-4}$    &  2.00   & 8   &  $7.9319 \times 10^{-4}$   &  2.01   \\
       16 &  $2.4091 \times 10^{-4}$    &  2.00   & 16  &  $1.9788 \times 10^{-4}$   &  2.00  \\
       32 &  $6.0207 \times 10^{-5}$    &  2.00   & 32  &  $4.9443 \times 10^{-5}$   &  2.00    \\
       Theory &   & 2.00     &  &   &  2.00     \\
      \bottomrule
    \end{tabular}
\end{table}
\vskip 2mm

}}

\section{Concluding remarks} \label{Sec6}

In this work, we investigate the subdiffusive Black-Scholes model of variable exponent, which is widely applied in situations where the historical information significantly influences the current pricing. A series of transformations are applied to convert the original model into a feasible formulation such that the numerical methods could be designed and analyzed. Numerical experiments are conducted to confirm the theoretical results. 

{\color{black}It is worth mentioning that, despite successful applications of the subdiffusive Black-Scholes model of variable exponent \cite{Luc}, the derivation of the variable-exponent model from certain stochastic dynamics requires further investigation. To be specific, it is shown in, e.g. \cite{Jum0}, that the subdiffusive Black-Scholes equation with a constant exponent $2H$ for some $0<H<1/2$ could be derived from the stochastic dynamics driven by the  fractional Brownian motion $B^H_t$ with the Hurst index $H$, which is a centered Gaussian process determined by the following covariance function
\begin{align}
\mathbb E[B^H_t\cdot B^H_t]=\frac{1}{2}\{|t|^{2H}+|s|^{2H}-|t-s|^{2H}\},~~\forall s,t\geq 0.
\end{align}
 Nevertheless, it is pointed out in \cite{Ara} that ``the assumption of a constant Holder regularity (Hurst exponent) in financial time-series seems to be too rigid to address some particularities of a market beyond tranquil period'', and there exist several works considering the time-varying Hurst index in stock markets \cite{Alv,Cao}. In particular, \cite{Ara} adopts the multifractional Brownian motion $W_{h(t)}$ for some function $h(t)\in (0,1)$, which is also a centered Gaussian process formally defined by its covariance 
\begin{align}
\mathbb E[W_{h(t)}\cdot W_{h(s)}]=D(h(t),h(s))\{|t|^{h(t)+h(s)}+|s|^{h(t)+h(s)}-|t-s|^{h(t)+h(s)}\},~~\forall s,t\geq 0,
\end{align}
where
$$D(x,y):=\frac{\sqrt{\Gamma(2x+1)\Gamma(2y+1)\sin(\pi x)\sin(\pi y)}}{2\Gamma(x+y+1)\sin(\pi(x+y)/2)}, $$
to derive a multifractional Black-Scholes model and substantiate its effectiveness in comparison with its  counterparts based on standard and fractional Brownian motions. 
As the exponent in the subdiffusive Black-Scholes model is closely related to the Hurst index as shown above, we adopt the variable exponent in (\ref{FBS}) following the spirit of the multifractional Brownian motion, and we will work on a rigorous proof from the stochastic dynamics with time-varying Hurst index to the variable-exponent fractional model in the future. }

\acknowledgements{\rm This work was partially supported by the National Social Science Foundation of China under Grant 24BTJ006.}

\vskip 0.3cm {\bf Conflict of Interest}\quad The authors declare no
conflict of interest.




\end{document}